\documentclass[10 pt]{article}
\usepackage[margin=1in]{geometry}

\usepackage{geometry}
\usepackage{graphicx}
\usepackage[ruled]{algorithm2e}
\usepackage{tikz}
\usetikzlibrary{shapes,shadows,arrows}
\usepackage{amsmath,amsfonts,amssymb,amscd,amsthm,xspace}
\usepackage{amsmath, color}
\theoremstyle{plain}
\newtheorem{example}{Example}[section]
\newtheorem{theorem}{Theorem}[section]

\newtheorem{lemma}{Lemma}[section]
\newtheorem{proposition}{Proposition}[section]

\theoremstyle{definition}
\newtheorem{definition}{Definition}[section]
\theoremstyle{remark}

\begin{document}
\title{\itshape A Novel Approach for Parameter and Differentiation Order Estimation for a Space Fractional Advection Dispersion Equation }
\author{{\small Abeer Aldoghaither$^1$, Taous-Meriem Laleg-Kirati$^1$ and Da-Yan Liu$^2$ }\\  {\small E-mail: abeer.aldoghaither@kaust.edu.sa, taousmeriem.laleg@kaust.edu.sa and dayan.liu@insa-cvl.fr  } \\ {\small $^1$ Mathematical and Computational Sciences and Engineering Division, King
Abdullah University of Science} \\ {\small and Technology (KAUST), Thuwal, Kingdom of
Saudi Arabia} \\{\small $^2$INSA Centre Val de Loire, Universit\'{e} d'Orl\'{e}ans, PRISME EA 4229, Bourges, France}}
\date{}

\maketitle

\begin{abstract}
In this paper, we propose a new approach, based on the so-called modulating functions to estimate the average velocity, the dispersion coefficient and the differentiation order in a space fractional advection dispersion equation. First, the average velocity and the dispersion coefficient are estimated by applying the modulating functions method, where the problem is transferred into solving a system of algebraic equations. Then, the modulating functions method combined with Newton's method is applied to estimate all three parameters  simultaneously. Numerical results are presented with noisy measurements to show the effectiveness and the robustness of the proposed method.

\subsubsection*{keywords} space fractional advection dispersion equation; modulating functions; parameter estimation 

\end{abstract}

\section{Introduction}
Fractional calculus has proven its efficiency in modeling many physical phenomena,  due to its memory and hereditary properties \cite{Fomin}, \cite{Schumer2}. It has been shown that it is appropriate to use fractional models to describe anomalous diffusion process, such as contaminants transport in soil, oil flow in porous media,  and groundwater flow  \cite{Wei, Schumer1,Xiong}, since it can capture some important features of particles transport, such as particles with velocity variation and long-rest periods \cite{Schumer2}.

In this paper, we are interested in identifying  the average velocity, the dispersion coefficient and the differentiation order for a space fractional advection dispersion equation using the measurements  of the concentration and the flux at final time. This is an inverse problem, where we use measurable data  to observe some properties of the structure of a physical system. Since the space fractional advection dispersion equation is usually used to model underground water transport in heterogeneous porous media, identifying parameters for such an equation is important  to understand how chemical or biological contaminates are transported throughout surface aquifer system \cite{Schumer1}. For instance, an estimate of the differentiation order in a ground water contaminant transport model can provide information about soil properties, such as the heterogeneous of the media \cite{Schumer2}.



Estimating coefficients for fractional differential equations is not a trivial problem. Moreover, the problem becomes more challenging when it involves the identification of the differentiation order, where usually using standard optimization approaches fails. Some progresses have been made on the parameter identification for the fractional diffusion equation \cite{SakamotoI}, \cite{ChengU}, \cite{JinA}, \cite{BondarenkoG}, \cite{MillerC}, while the parameter identification problem for the space fractional advection dispersion equation has not been paid much attention. Zhang \textit{et al.} \cite{Zhang}, considered an inverse problem based on optimization for simultaneous identification of multi parameters in a space fractional advection dispersion equation. This inverse problem has been solved numerically using the optimal perturbation regularization algorithm introduced by Chi \textit{et al.}  \cite{Chi}. However,  the sensibility of the algorithm depends on the regularization parameter, the numerical differential step, and the initial iteration. In their study, they considered the  Tikhonov regularization to treat the ill posedness of the problem. In spite of the importance, to the best of our knowledge, except \cite{Zhang} there is no published work on the parameter identification problem for the space fractional advection dispersion equation.

In our study, we propose a novel approach, where the modulating functions method is combined with the first order Newton's method to estimate all three parameters simultaneously.

The modulating functions method was first suggested in 1954 by Shinbrot \cite{ShinbrotO}, then it has been used in many applications, such as  signal processing and control theory (see, e.g. \cite{m7}, \cite{m8}, \cite{SadabadiS}, \cite{FedeleA},  \cite{Liu2011d}, \cite{LiuI}).
For instance, Fedele  \textit{et al.} \cite{FedeleA}, presented a recursive frequency estimation scheme based on trigonometric and spline-type modulating functions. Janiczek \cite{JaniczekG},  generalized the modulating functions method to the fractional differential equations, where he aimed to reduce the fractional order to an integer order in noise free case. Liu \textit{et al.} \cite{LiuI}, applied the modulating functions method to identify unknown parameters for a class of fractional order linear systems. Furthermore, Sadabadi \textit{et al.}  \cite{SadabadiS}, estimated the parameters for a two dimensional continuous-time system, based on a two dimensional modulating functions approach. However, the method was not generalized to fractional partial differential equations.

The main goal of this paper is to introduce a new effective  modulating functions based approach which estimates simultaneously the coefficients and the differentiation order for a space fractional advection dispersion equation.
The proposed approach has several  advantages. Firstly, the problem is transferred into solving a system of algebraic equations. Then, the estimations of the unknown parameters can be exactly given by integral formula. Secondly, initial values, which are usually unknown in most real life applications, are not required in this approach. Thirdly, instead of  computing the fractional derivative of the solution of the partial differential equation, the fractional derivatives of the modulating functions are computed, which can be simpler, if the modulating functions are well-chosen. Fourthly, a regularization technique is not needed since the proposed estimations involving integral formula are robust and can help to reduce the effect of noise (see \cite{Liu2011d}, \cite{LiuI}). Furthermore, the proposed approach simplifies the optimization problem, where the problem is only optimized with respect to one parameter, instead of all parameters.

The paper is organized as follows: in Section 2, we introduce the considered problem and some primary definitions. In Section 3, we apply the modulating functions method to the space fractional advection dispersion equation, where the average velocity and the dispersion coefficient are estimated by solving a system of algebraic equations. Then, the modulating functions method is combined with the first order Newton's method to estimate all three parameters simultaneously, in Section 4. In Section 5, we give some numerical results  to show the efficiency and robustness of the method. A discussion is given  in Section 6. Finally, a conclusion summarizes the obtained results.

\section{Preliminaries}
In this section, we recall the definition of the so-called modulating functions and  some useful properties.

\begin{definition}(\cite{podl} p. 62) The  $\alpha^{th}$ order Riemann-Liouville fractional derivative of a continuous function $f$ defined on $\mathbb{R}$, with $\alpha \in \mathbb{R}$, is defined as follows: $\forall t \in  \mathbb{R} $, $ n \in \mathbb{N}^*$,
				\begin{eqnarray}\label{R_L_left}
					\begin{array}{lll}
						 D^\alpha_x f(x) = \displaystyle\frac{1}{\Gamma{(n-\alpha)}} \frac{d^n}{dt^n}  \int_0^x  (x-\tau) ^{n-\alpha-1} f(\tau)  d \tau, &   n-1\le \alpha < n.
					\end{array}
				\end{eqnarray}
\end{definition}

\begin{definition}\cite{Preisg}
A function $\phi(x) \in C^n$, defined over the interval $[a,b]$, is called a modulating function of order $k$  with $k \in N^*$  if:
\begin{eqnarray}\label{condition}
\begin{array}{cc}
\phi^{(i)}(a)= \phi^{(i)}(b) = 0,   & i=0,1,\dots,k-1.
\end{array}
\end{eqnarray}
\end{definition}
The following lemma describes a useful generalized integration by parts formula. This lemma was obtained by applying the convolution theorem of the Laplace transform \cite{LiuI}.

\begin{lemma}\label{lemma1}
\cite{LiuI}
If the $\alpha^{th}$ order Riemann-Liouville derivative of f exists where $n-1 \le \alpha < n$, and g is an $n^{th}$ order  modulating function defined on $[0,L]$. Then, we have:
\begin{eqnarray}\label{propo1}
\int_0^L g(L-x) D_x^\alpha f(x) d x =  \int_0^L D_x^\alpha g(x) f(L-x) d x.
\end{eqnarray}
\end{lemma}

\section{Problem Statement}
We consider the following space fractional advection dispersion equation with initial and Dirichlet boundary conditions: for any $0 < x < L$ and $t>0$,
\begin{equation}\label{FADE}
		\left\{
			\begin{array}{ll l}  	
					\dfrac{\partial c(x,t)}{\partial t} &=& -\nu \dfrac{\partial c(x,t)}{\partial x} + d \dfrac{\partial ^ \alpha c(x,t) }{\partial x^\alpha} + r(x,t),   \\		
					c(x,0)&=&g_0(x),    \\
     					c(0,t)&=&0, \\
     					c(L,t)&=&0, \\
			\end{array}
		\right.
\end{equation}
where $c$ is the solute concentration, $\nu$ is the average velocity, $d$ the is dispersion coefficient, $r$ is the source term, $\alpha$ is the differentiation order for the space derivative with $1 < \alpha \le 2 $. We assume that $\nu$ and $d$ are constants and the fractional derivative is a Riemann-Liouville  derivative of order $\alpha$.

The inverse problem falls in the category of parameter identification. We would like to estimate the unknown parameters $d, \nu$, and $\alpha$. We suppose that the concentration $c$ and the flux $\frac{\partial c}{\partial t}$ are unknown except at a final time $t=T$, where we can measure them:
\begin{eqnarray}\label{prior1}
			\begin{array}{lll}
				c(x,T)=g_1(x), &\displaystyle\frac{\partial c(x,T)}{\partial t}=g_2(x),&  0 < x < L. \\
			\end{array}
	\end{eqnarray}
	
	\section{Modulating Functions Method for estimating the average velocity and the dispersion coefficient}
In this section, we present our first result, where the modulating functions method is applied to estimate  the average velocity and the dispersion coefficient by assuming that the differentiation order is known.

\begin{theorem}\label{theorem1}
Let $\{\phi_n(x)\}_{n=1}^N$  be a set of $n^{th}$ order modulating functions defined on the interval $[0,L_1]$ where $2 \leq N$ and $L_1 \leq L$, then the solution of the following linear system gives the estimations  of the parameters  $\nu$ and $d$:
 \begin{eqnarray}\label{sys1}
 \begin{pmatrix}
  A_1		&	B_1  \\

A_2 &  B_2\\
\vdots & \vdots \\
A_N & B_N

 \end{pmatrix}
 \begin{pmatrix}
 -\nu	  \\
  \\
d \\
 \end{pmatrix} =
  \begin{pmatrix}
 C_1  \\
C_2\\
\vdots
\\
C_N
 \end{pmatrix},
\end{eqnarray}
where
\begin{eqnarray}
A_n = \int_0^{L_1}  \frac{\partial \phi_n(L_1-x)}{\partial x}   c(x,t) d x,
\end{eqnarray}
\begin{eqnarray}
B_n=   \int_0^{L_1} \dfrac{\partial ^ \alpha \phi_n(x)  }{\partial x^\alpha} c(L_1-x,t) d x,
\end{eqnarray}
\begin{eqnarray}
C_n = \int_0^{L_1} \phi_n(L_1-x) \dfrac{\partial c(x,t)}{\partial t}  -  r(x,t) \phi_n(L_1-x) d x.
\end{eqnarray}

\end{theorem}
\begin{proof}

\underline{Step 1:}
Multiply (\ref{FADE}) by the modulating  functions $\phi_n(L_1-x)$  for $n=1,\dots,N$, then we get:
\begin{equation}
\begin{split}
& \phi_n(L_1-x) \dfrac{\partial c(x,t)}{\partial t} = \\ &\qquad \qquad - \nu \dfrac{\partial c(x,t)}{\partial x} \phi_n(L_1-x) +  d  \dfrac{\partial ^ \alpha c(x,t) }{\partial x^\alpha} \phi_n(L_1-x) +  r(x,t) \phi_n(L_1-x).
\end{split}
\end{equation}
\underline{Step 2:}
Integrating over the interval $[0,L_1]$, gives us:
\begin{equation}\label{inte_part} 
\begin{split}
  &\nu \int_0^{L_1} - \dfrac{\partial c(x,t)}{\partial x}  \phi_n(L_1-x) d x + d \int_0^{L_1} \dfrac{\partial ^ \alpha c(x,t) }{\partial x^\alpha} \phi_n(L_1-x) d x =  \\ & \qquad \qquad \qquad \qquad\int_0^{L_1} \phi_n(L_1-x) \dfrac{\partial c(x,t)}{\partial t} d x - \int_0^{L_1} r(x,t) \phi_n(L_1-x) d x.
  \end{split}
\end{equation}
\underline{Step 3:} By applying integration by parts and Lemma \ref{lemma1} to the left-side of equation (\ref{inte_part}), we obtain:
\begin{equation}\label{mf1} 
\begin{split}
 & \nu \int_0^{L_1}  -\dfrac{\partial \phi_n(L_1-x) }{\partial x}  c(x,t) d x + d \int_0^{L_1} \dfrac{\partial ^ \alpha \phi_n(x)  }{\partial x^\alpha} c(L_1-x,t) d  x =\\ &\qquad \qquad \qquad \qquad \int_0^{L_1} \phi_n(L_1-x) \dfrac{\partial c(x,t)}{\partial t} d x - \int_0^{L_1} r(x,t) \phi_n(L_1-x) d x,
  \end{split}
\end{equation}
where the boundary conditions are eliminated by the properties of the used modulating functions.

Finally, the unknown parameters can be estimated by solving the linear system given in (\ref{sys1}).
\end{proof}

\section{Parameter and differentiation order estimation}
In this section, we present a new approach, where we  combine a first order Newton's method and the modulating functions method to simultaneously, estimate $\nu$, $d$, and $\alpha$.

Before presenting the second result of this paper, let us introduce the following proposition.

\begin{proposition}\label{p1}
Let $\{\phi_n(x)\}_{n=1}^N$  be a set of $n^{th}$ order modulating functions defined on the interval $[0,L_1]$ where $2 \leq N$ and $L_1 \leq L$. Then the parameters $d$, $\nu$ can be written in terms of $\alpha$ using Theorem \ref{p1}, and the following linear system estimates the derivatives of $d$ and $\nu$ with respect to $\alpha$:
\begin{eqnarray}\label{d'v'}
 \begin{pmatrix}
  \hat{A_1}		&	\hat{B_1}  \\

\hat{A_2} &  \hat{B_2}\\
\vdots & \vdots
\\
\hat{A_N} & \hat{B_N}

 \end{pmatrix}
 \begin{pmatrix}
 \frac{\partial d(\alpha)}{\partial \alpha}	  \\
  \\
- \frac{\partial \nu(\alpha)}{\partial \alpha}\\
 \end{pmatrix} =
  \begin{pmatrix}
 \hat{C_1}  \\
\hat{C_2}\\
\vdots \\
\hat{C_N}
 \end{pmatrix},
\end{eqnarray}
where
\begin{eqnarray}
 \hat{A_n} =  \int_0^{L_1} \dfrac{\partial ^ \alpha \phi_n(x)  }{\partial x^\alpha} u(L_1-x,t) dx,
 \end{eqnarray}
 \begin{eqnarray}
 \hat{B_n}=\int_0^{L_1}  \dfrac{\partial \phi_n(L_1-x) }{\partial x}  u(x,t) dx,
 \end{eqnarray}
  \begin{eqnarray}
\hat{C_n}=  -d(\alpha) \int_0^{L_1} \frac{\partial}{\partial \alpha}\dfrac{\partial ^ \alpha \phi_n(x)  }{\partial x^\alpha} u(L_1-x,t) dx.
\end{eqnarray}
\end{proposition}

\begin{proof}
This proof can be obtained by differentiating (\ref{mf1}) with respect to $\alpha$.
\end{proof}

In the next subsection, we present  the second result, where we estimate the average velocity, the dispersion coefficient, and the differentiation order for equation (\ref{FADE}), using the measurements given in (\ref{prior1}).

\subsection{Combined Newton's and Modulating Functions Method to Estimate $d$, $\nu$, and $\alpha$}

Order of differentiation is unknown and often challenging to estimate. However, by using the concept of modulating functions, this difficulty is greatly reduced. Due to the nature of the problem it is demanding that we split the solution algorithm into two stages: the first stage solves the pervious problem, and the second stage deals with the optimization problem. Applying the modulating functions greatly simplifies the second stage of the algorithm by reducing the number of unknown parameters in the optimization problem.

Now, we introduce the two stage algorithm to estimate $\nu$, $d$ and $\alpha$.

 \underline{Stage 1:} In this stage, we apply Theorem \ref{theorem1} to re-write $\nu$ and $d$ as functions of the unknown $\alpha$: $\nu(\alpha)$ and $d(\alpha)$.
%

Then, we consider the following equation:
\begin{equation}\label{}
    \dfrac{\partial c(x,t)}{\partial t} = -\nu(\alpha)\ \dfrac{\partial c(x,t)}{\partial x} + d(\alpha)\ \dfrac{\partial ^ \alpha c(x,t) }{\partial x^\alpha} + r(x,t).
\end{equation}
If $\phi_m$ is an $n^{th}$ order modulating function on $[0,L_1]$, then using a similar way of obtaining (\ref{mf1}), we get:
 \begin{equation}\label{dddv} 
 \begin{split}
  &\nu(\alpha) \int_0^{L_1}  -\dfrac{\partial \phi_m(L_1-x) }{\partial x}  c(x,t) d x + d(\alpha) \int_0^{L_1} \dfrac{\partial ^ \alpha \phi_m(x)  }{\partial x^\alpha} c(L_1-x,t) d x  = \\ & \qquad\qquad\qquad\qquad\int_0^{L_1} \phi_m(L_1-x) \dfrac{\partial c(x,t)}{\partial t} d x - \int_0^{L_1} r(x,t) \phi_m(L_1-x) d x.
  \end{split}
\end{equation}

Since $\alpha$ is the only unknown in equation (\ref{dddv}), we can write it as follows:
\begin{eqnarray}
 K(\alpha) = U,
\end{eqnarray}
 where
 \begin{equation}\label{K}
K(\alpha) := d(\alpha) \int_0^{L_1} \dfrac{\partial ^ \alpha \phi_m(L_1-x)  }{\partial x^\alpha} c(x,t)  d x  -   \nu(\alpha) \int_0^{L_1}  \dfrac{\partial \phi_m(L_1-x) }{\partial x}  c(x,t) d x,
\end{equation}
and
\begin{eqnarray}
U:=\int_0^{L_1} [\phi_m(L_1-x) \dfrac{\partial c(x,t)}{\partial t}  - r(x,t) \phi_m(L_1-x) ] d x.
\end{eqnarray}

\underline{Stage 2:}
In this stage, the inverse problem is formulated as the following minimization problem of the $L_2$-norm of the output error with respect to only one unknown $\alpha$:
\begin{eqnarray}\label{cost_1}
J(\alpha) = \parallel K(\alpha) - U \parallel_2^2,
\end{eqnarray}

Then, a first order Newton's type method is used to solve the minimization problem. At each iteration, the order $\alpha$ is updated using:
 \begin{eqnarray}
 \begin{array}{cc}
 K(\alpha_k)-U = \Delta \alpha_k K'(\alpha_k),
 \end{array}
 \end{eqnarray}
where
 \begin{eqnarray} \Delta \alpha_k=\alpha_{k+1}-\alpha_k,
  \end{eqnarray}
and the gradient $K'(\alpha)$ is computed using the next proposition.
\begin{proposition}
Using $\nu(\alpha)$ and $d(\alpha)$ which are the estimations given by Theorem \ref{theorem1} and $K(\alpha)$ is given in (\ref{K}), the gradient $K'(\alpha)$ can be computed as follows:
\begin{equation}
\begin{split}
&K'(\alpha)= d'(\alpha) \int_0^{L_1} \frac{\partial ^ \alpha \phi_m(x)  }{\partial x^\alpha} u(L_1-x,t) dx -   \nu'(\alpha) \int_0^{L_1}  \frac{\partial \phi_m(L_1-x) }{\partial x}  u(x,t) dx \\ &\qquad\qquad\qquad + d(\alpha) \int_0^{L_1} \frac{\partial}{\partial \alpha}\frac{\partial ^ \alpha \phi_m(x)  }{\partial x^\alpha} u(L_1-x,t) dx,
\end{split}
\end{equation}
where $d'(\alpha)$ and $\nu'(\alpha)$ are given by Proposition \ref{p1}.
\end{proposition}
\begin{proof}
This proof can be completed by  differentiating (\ref{K}) with respect to $\alpha$.
\end{proof}

	Newton's iteration is a gradient based method and we know that for any gradient based algorithm, most of the computational effort is spent on computing the gradient at each step. However, here, thanks to the modulating functions method we have analytical closed form of gradient. This analytical form is more stable and requires less computational power. Further, using a similar way,  we can efficiently compute higher order derivatives.

For the convenience, we present a description of the proposed algorithm below:
\begin{center}
{\scriptsize
\tikzstyle{decision} = [diamond, draw, fill=white!30]
\tikzstyle{line} = [draw, -stealth, thick]
\tikzstyle{elli}=[draw, ellipse, fill=white!30,minimum height=8mm, text width=5em, text centered]
\tikzstyle{block} = [draw, rectangle, fill=white!30, text width=9em, text centered, minimum height=8mm, node distance=10em]
\tikzstyle{block1} = [draw, diamond, fill=white!30, text width=9em, text centered, minimum height=8mm, node distance=5em]
\begin{tikzpicture}
\node [block] (start) {$d(\alpha_k)$, $\nu(\alpha_k)$};
\node [block, left of=start, xshift=-2em] (process1) {$\alpha_{k+1}=\alpha_k-\frac{K(\alpha_k)-U}{K'(\alpha_k)}$};
\node [elli, above of=start, yshift=3em] (user) {initial value $\alpha_o$};
\node[block, below of=start, yshift=2em](decision1){$\| K(\alpha_k)-U\|^2 < \epsilon$};
\node[elli, below of=decision1, yshift=-2em](decision2){$\alpha, d, \nu$};
\path [line] (user) -- (start);
\path [line] (start) -- (decision1);
\path [line] (decision1) -| node[yshift=0.7em, xshift=3em] {no} (process1);
\path [line] (process1) --(start);
\path [line] (decision1) -- node[yshift=0em, xshift=1em] {yes}(decision2);
\end{tikzpicture}
}
\begin{algorithm}
    \caption{}
  Step1: Start with an initial guess $\alpha_o$.
 \\  Step 2: Compute the corresponding $d(\alpha_k)$, $\nu(\alpha_k)$.
   \\ Step 3: Compute $\| K(\alpha_k)-U\|^2$,

    {\eIf{$\| K(\alpha_k)-U\|^2< \epsilon$} {
 output: $\nu(\alpha_k)$, $d(\alpha_k)$ and $\alpha_k$ }
   {update  $\alpha_{k+1}=\alpha_k-\frac{K(\alpha_k)-U}{K'(\alpha_k)}$ and go back to step 2.} }
\end{algorithm}

\end{center}
\section{Numerical results}

In this section, we present some numerical results to show the efficiency and the robustness of the presented method.

First, we estimate $\nu$ and $d$ by solving the system given in (\ref{sys1}). Then, we use the algorithm given in Section 5 to estimate the parameters $\nu$, $d$, and $\alpha$ on a finite interval from noisy measurements. The value of $T$ is taken at the time where we estimate our parameters. We consider the following  polynomial modulating functions whose fractional derivatives are simple to calculate:
$
\phi_n(x)= x^{N+b+1-n} (L_1-x)^{b+n},
$
where $L_1\le L$, $b=3$, which is the smallest value that satisfies (\ref{condition}) and $n=1,2,\dots,N$, where $N$ is the number of modulating functions. Moreover, we apply the trapezoidal rule to numerically approximate the integrals.

\begin{example}
Let us consider the following space fractional advection-dispersion equation:
	\begin{eqnarray}\label{example_1}
		\begin{array}{lll}
			 \dfrac{\partial c(x,t)}{\partial t} = -\nu \dfrac{\partial c(x,t)}{\partial x} + d \dfrac{\partial ^ \alpha c(x,t) }{\partial x^\alpha} + r(x,t), & 0<x< 9, & t >0,
		\end{array}
 	\end{eqnarray}
	 with the following initial and Dirichlet boundary conditions:
\begin{eqnarray}\label{condition_1}
	\left\{
		\begin{array}{ lll }  	
     			c(x,0)&=&x(9-x),   \\
     			c(0,t)&=&0, \\
     			c(9,t)&=&0,\\
		\end{array}
	\right.
\end{eqnarray}
	where
\begin{equation}
	r(x,t)= \left\{
		\begin{array}{ ll }  	
     			 \cos(-t)[0.2 (9-2x) - (\frac{9\Gamma(2)}{\Gamma (2-\alpha)} x^{1-\alpha}- \frac{\Gamma(3)}{\Gamma (3-\alpha)} x^{2-\alpha})]+ \\  \qquad\qquad\qquad \sin(-t) x(x-9),   & 0< x \le 9,\\
     			0,  &x=0.\\
		\end{array}
	\right.
\end{equation}
 The exact solution of the forward problem is  $c(x,t)=\cos(-t) x(9 -x)$ and the flux is $\displaystyle\frac{\partial c(x,t)}{\partial t} = \sin(-t) x (9 - x)$.
\end{example}

\subsection*{\underline {Estimating $\nu$, $d$ when $\alpha$ is known}}
In this part, we assume that the differentiation order $\alpha$ is known and we estimate $\nu$ and $d$. We set the exact values of the average velocity as $\nu=0.2$, the dispersion coefficient $d=1$, the differentiation order $\alpha=1.8$, the final time $T=1$. In Figure \ref{fig3}, the estimated values of $\nu$, and $d$ when adding a $3\%$ white Gaussian noise to the measurements. Three modulating functions are used and we increase the length of the integration interval $[0,L_1]$. As we can see in Figure \ref{fig4}, the results are satisfactory and the relative error decreases as we increase the length of the integration interval $[0,L_1]$. Although not presented, we would like to note that the results obtained using different number of modulating functions are quite similar to those presented in Figure \ref{fig3}.
\begin{figure}[http]
 \begin{minipage}[b]{0.5\linewidth}
    \centering
    \includegraphics[width=\linewidth]{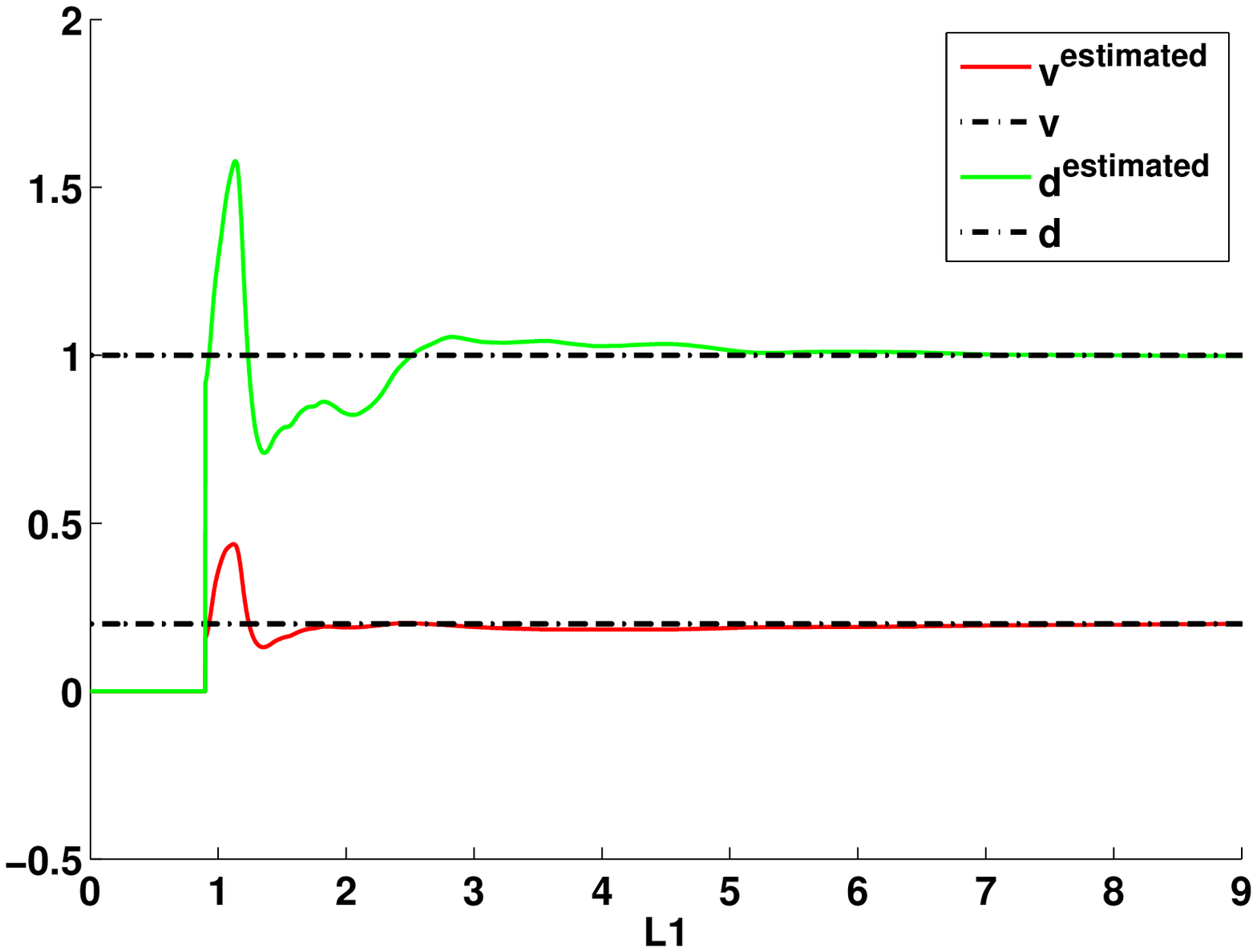}
    \caption{The estimated $d$ and $\nu$ with $3\%$ noise with different values of $L_1$ when $d=1$, $\nu=0.2$.}\label{fig3}
  \end{minipage}
  \hspace{0.5cm}
  \begin{minipage}[b]{0.5\linewidth}
    \centering
    \includegraphics[width=\linewidth]{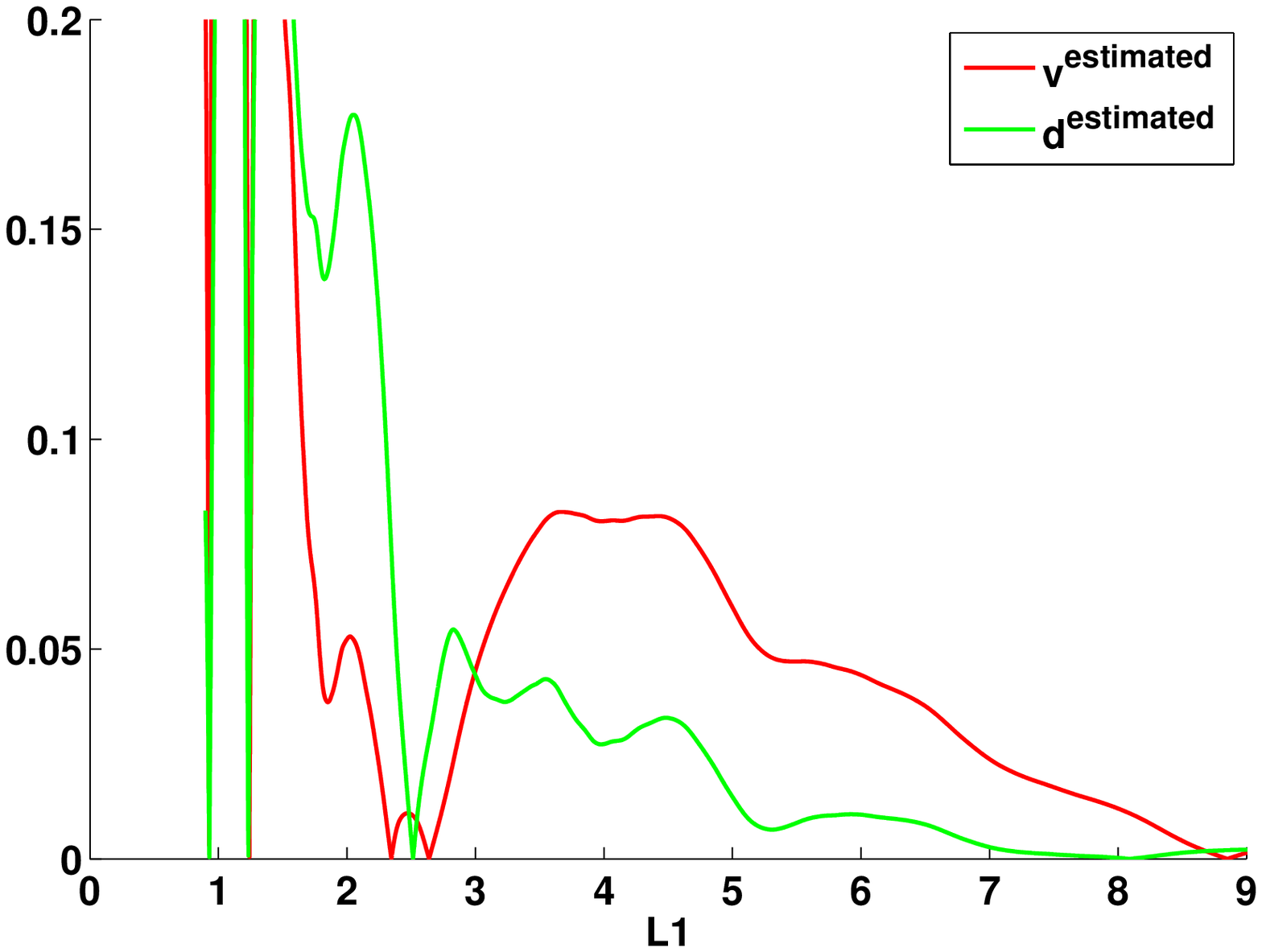}
    \caption{The relative errors of $\nu$ and $d$ with $3\%$ noise for different values of $L_1$.}\label{fig4}
  \end{minipage}
\end{figure}

\subsection*{\underline {Estimating $\nu$, $d$ and $\alpha$}}

In this part, we will use the combined Newton's and modulating functions method to estimate all three parameters simultaneously.  We set the exact values of the average velocity $\nu=0.5$, the dispersion coefficient $d=1$, the differentiation order $\alpha=1.8$, the final time $T=1$, the initial guess $\alpha_o=1.4$. Figure \ref{Histogram_best1} represents the estimated parameters using 7 modulating functions, where the noise level is $2\%$. From this figure, we observe that the numerical results are quite satisfactory,
where the relative error is less than $9\%$ when integrating over the interval $[0,4.5]$ and drops to less than $1\%$ as we increase the length of the integration interval.
\begin{figure}[http]
     \centering
    \includegraphics[width=100mm]{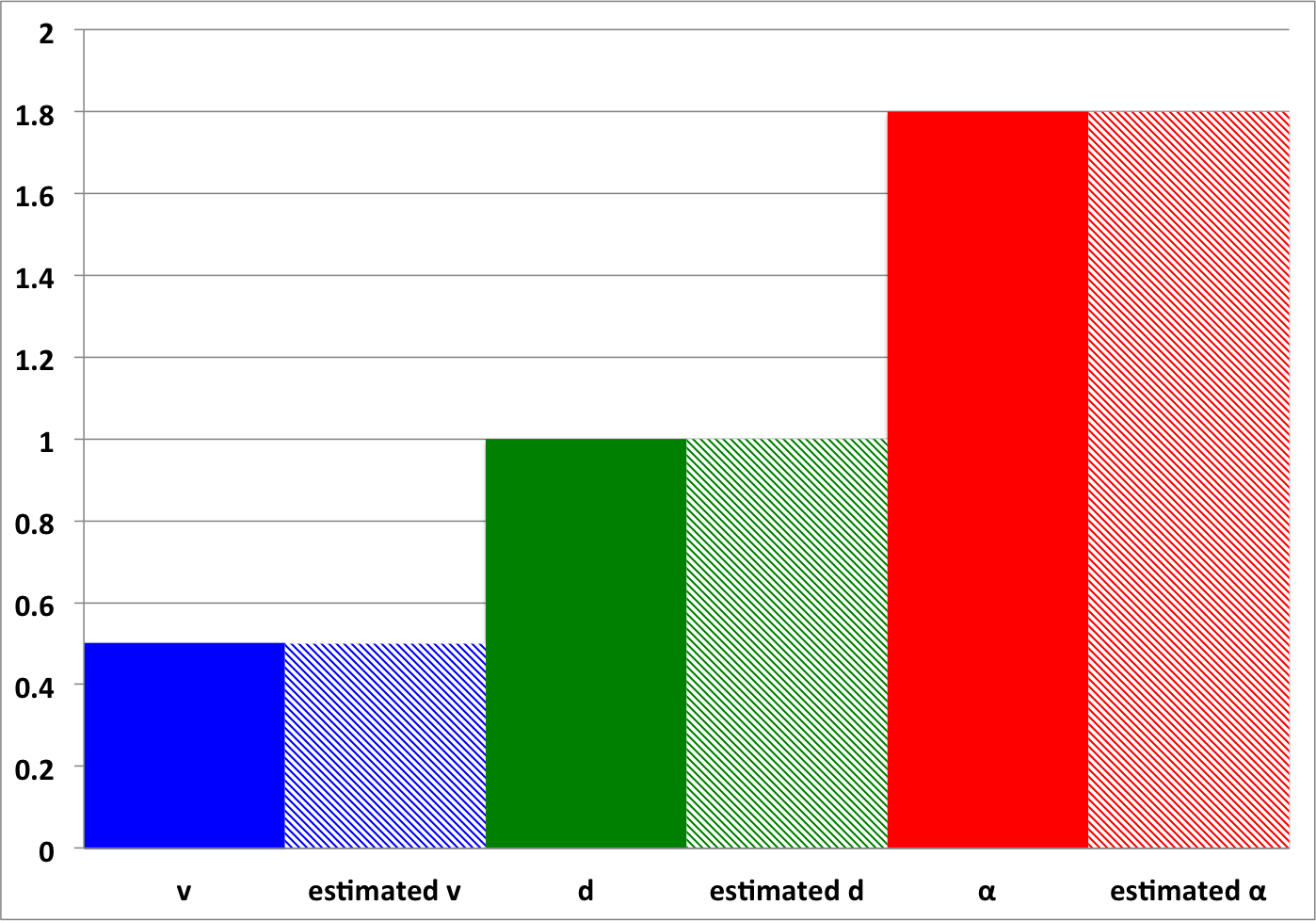}
    \caption{A comparison between the exact and the estimated parameters with 7 modulating functions when adding a $2\%$ noise to the measurements.}\label{Histogram_best1}
\end{figure}

In Figure \ref{Histogram_noise2}, the comparisons under different noise levels $1\%$, $3\%$, $5\%$, $10\%$, between the exact values of the parameters and the estimated values are given. From this figure, it can be seen that the results are stable and remain reasonable even when adding $10\%$ noise to the measurements. Figure \ref{Histogram1} represents the comparison with different number of modulating functions  when adding a $2\%$ white Gaussian noise to the measurements. Even with different number of modulating functions the errors are small and the results are quite satisfactory. This is confirmed in Table \ref{Tab_1}, where the relative errors are between $2 \times 10^{-4}$ and $4.81\times 10^{-3}$. However,
it is noted that the number of the modulating functions has an effect on the stability and the accuracy of the presented algorithm which will be discussed in the next section.

\begin{figure}[http]
     \centering
    \includegraphics[width=130mm]{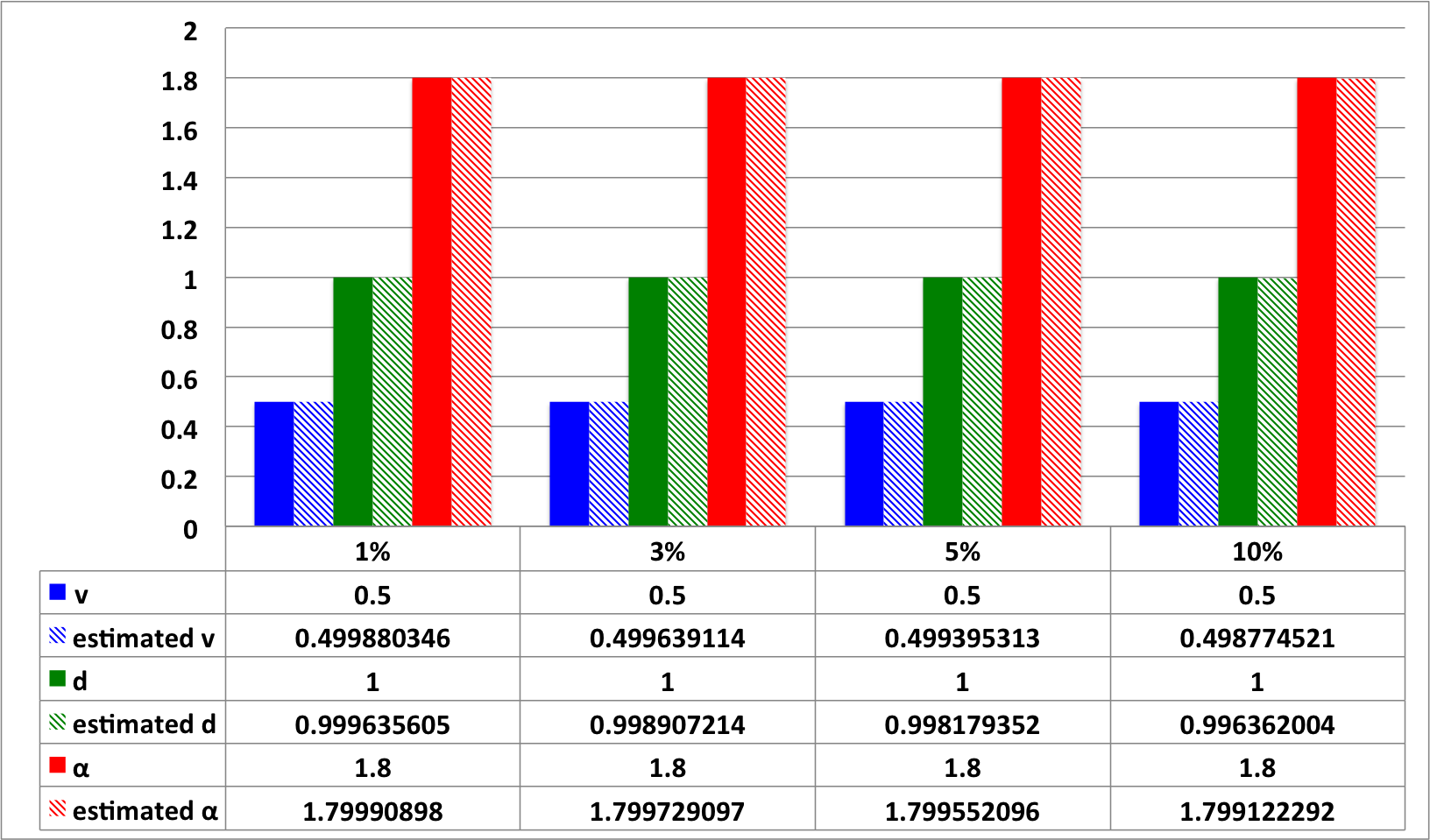}
    \caption{The estimated parameters obtained with 7 modulating functions and different noise levels. }\label{Histogram_noise2}
\end{figure}
\begin{figure}[http]
     \centering
    \includegraphics[width=130mm]{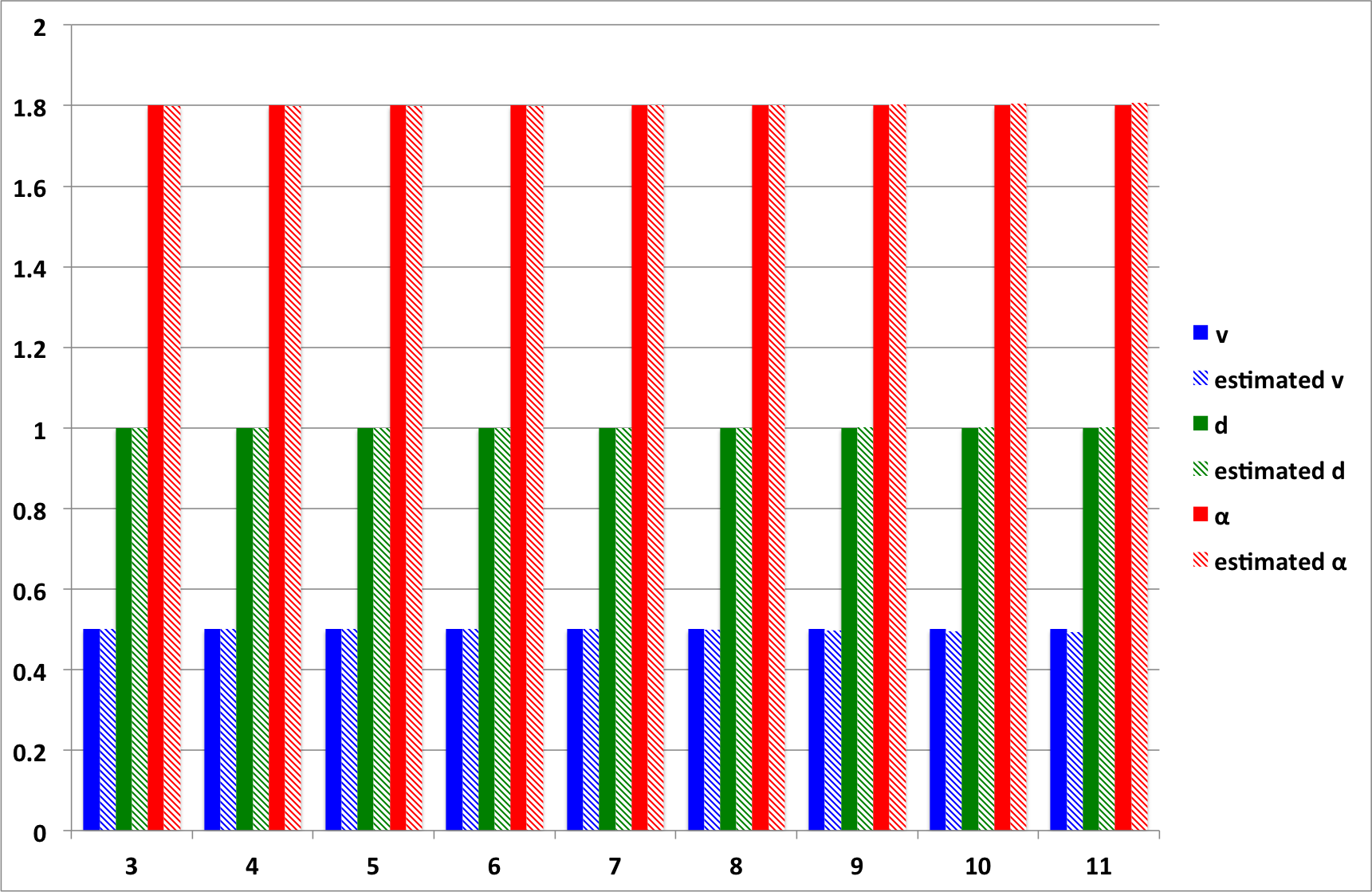}
    \caption{The estimated parameters with different number of modulating functions. }\label{Histogram1}
\end{figure}

 \begin{center}
 \begin{table}[http]
 \caption{$d=1$, $\alpha=1.8$, $\nu=0.5$, and $\Delta x=\frac{1}{3500}$, 2\% noise.}\label{Tab_1}
 \begin{center}
\begin{tabular}{c cc c  c c c c c}

\hline\hline
number of  & Estimated Value  & Relative Error  &  Relative error \\ 
 modulating & ($\nu$, $d$, $\alpha$)   &($\nu$, $d$, $\alpha$)\\
\hline
3 &  (0.5003, 0.9991,  1.7989)  & (5.39E-4, 9.17E-4, 6.37E-4) & 0.7E-3     \\

4 &	(0.5008, 0.9988, 1.7982)	&  (1.69E-3, 1.16E-3, 1.01E-3)&  1.1E-3     	\\

5 &       (0.5009, 0.9989, 1.7982)          & (1.90E-3, 1.12E-3, 1.01E-3)&    1.1E-3            \\

6 &        (0.5006, 0.9992, 1.7988)      &      (1.16E-3, 8.48E-3, 6.53E-3) &   0.7E-3         \\

7 &     (0.4998, 0.9996, 1.79999)       &          (4.43E-4, 4.19E-4,	 3.8E-06) &   0.2E-3    \\

8 &              (0.4986, 1.0001, 1.8016)          &    (2.79E-3, 6.17E-5, 8.73E-4) &   0.1E-2   \\

9 &     (0.4971, 1.0005, 1.8034)            &   (5.77E-3, 4.94E-4, 1.92E-3)&      2.1E-3         \\

10 &    (0.4954, 1.0008, 1.8055)        &     (9.24E-3, 7.89E-4, 3.06E-3) &      3.4E-3      \\

11&    (0.4935, 1.0009,  1.8076)      &      (1.30E-2, 8.83E-4, 4.24E-3)  & 4.8E-3  \\

\hline
\end{tabular}
\end{center}
\end{table}
\end{center}

\section{Discussion}
A Two-Stage algorithm has been used to estimate the parameters and the differentiation order for a fractional differential system.  In the proposed approach, we take advantage of the properties of the modulating functions to over come the difficulties in estimating the differentiation order. The efficiency and the robustness against corrupting noise  with different number of modulating functions have been confirmed by numerical examples. It is noted that the choice and the number of modulating functions can effect the accuracy of the proposed algorithm.
On the one hand, as we can see in Figures \ref{fig5} and \ref{fig6}, for 3 up to 20 modulating functions the relative errors vary, but still less than $3\%$. It is noted that the numerical accuracy becomes worse as the number of modulating functions increases (see Figures \ref{fig7}, \ref{fig8}).
This is because  the stability of the algebraic system given in (\ref{sys1}) depends on the modulating functions, as we increase the number of modulating functions the system of equations increases. Hence, when using polynomial modulating functions with a large system we will lose independency. In all cases, the results are stable and remain reasonable even for up to 20 modulating functions and the presented algorithm works well with reasonable number of modulating functions, which is further confirmed by the errors in Table \ref{Tab_1}.

\begin{figure}[http]
 \begin{minipage}[b]{0.5\linewidth}
    \centering
    \includegraphics[width=\linewidth]{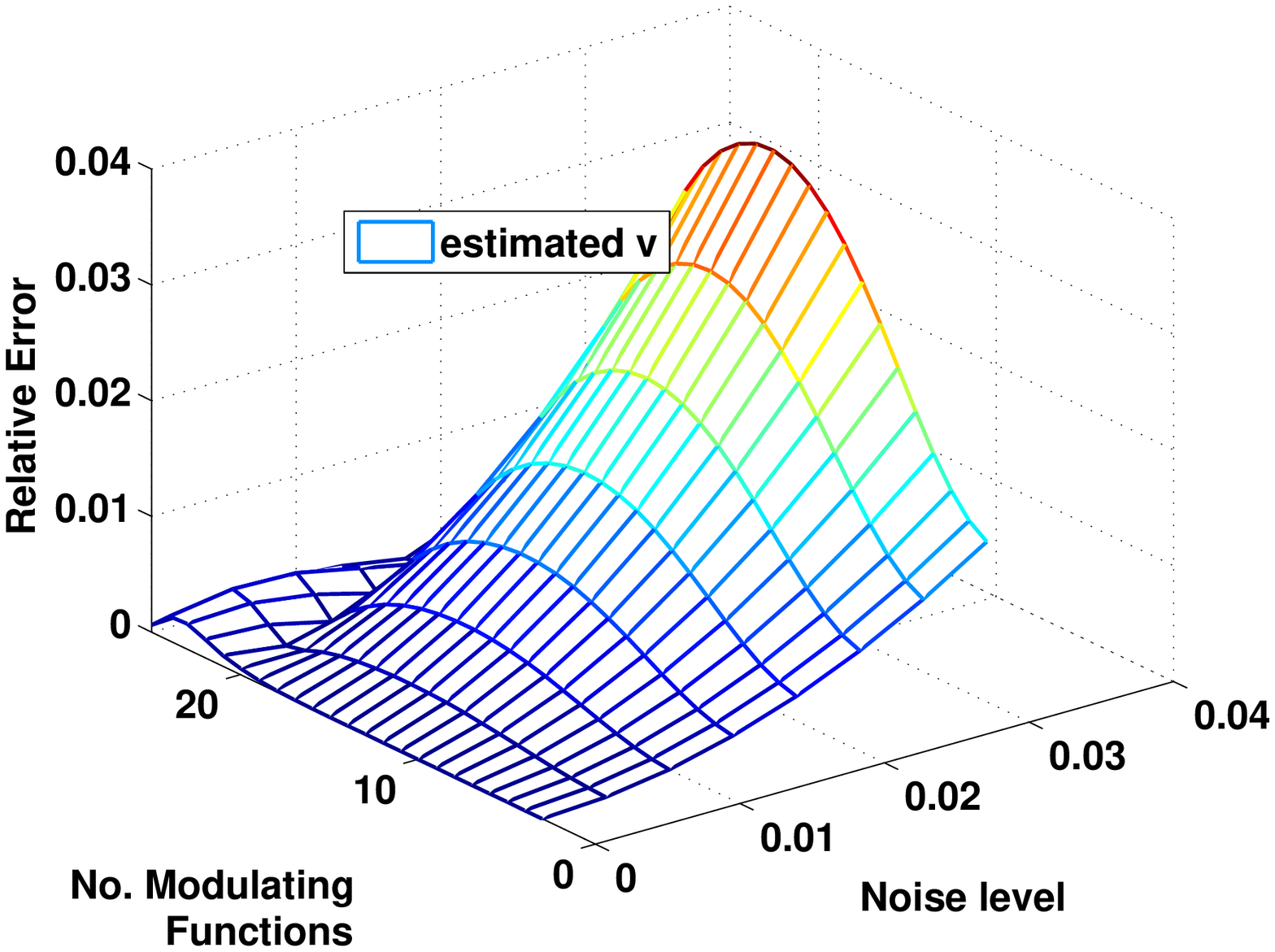}
    \caption{The relative errors of $d$ for different noise levels for up to 25 modulating functions.}
\label{fig5} 
  \end{minipage}
  \hspace{0.5cm}
  \begin{minipage}[b]{0.5\linewidth}
    \centering
    \includegraphics[width=\linewidth]{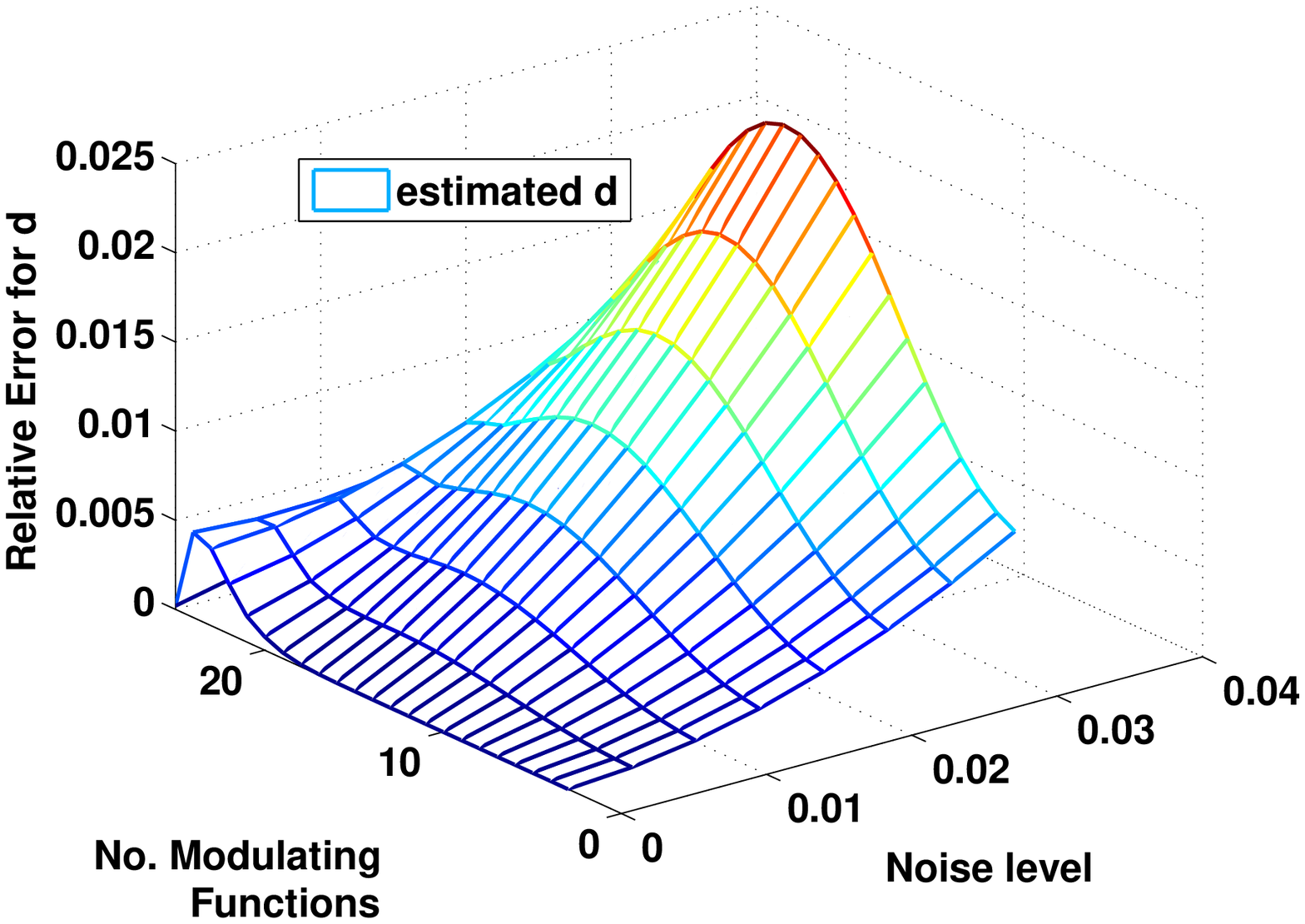}
    \caption{The relative errors of $\nu$ for different noise levels for up to 25 modulating functions..}\label{fig6}
  \end{minipage}
\end{figure}

\begin{figure}[http]
\begin{minipage}[b]{0.5\linewidth}
    \centering
    \includegraphics[width=\linewidth]{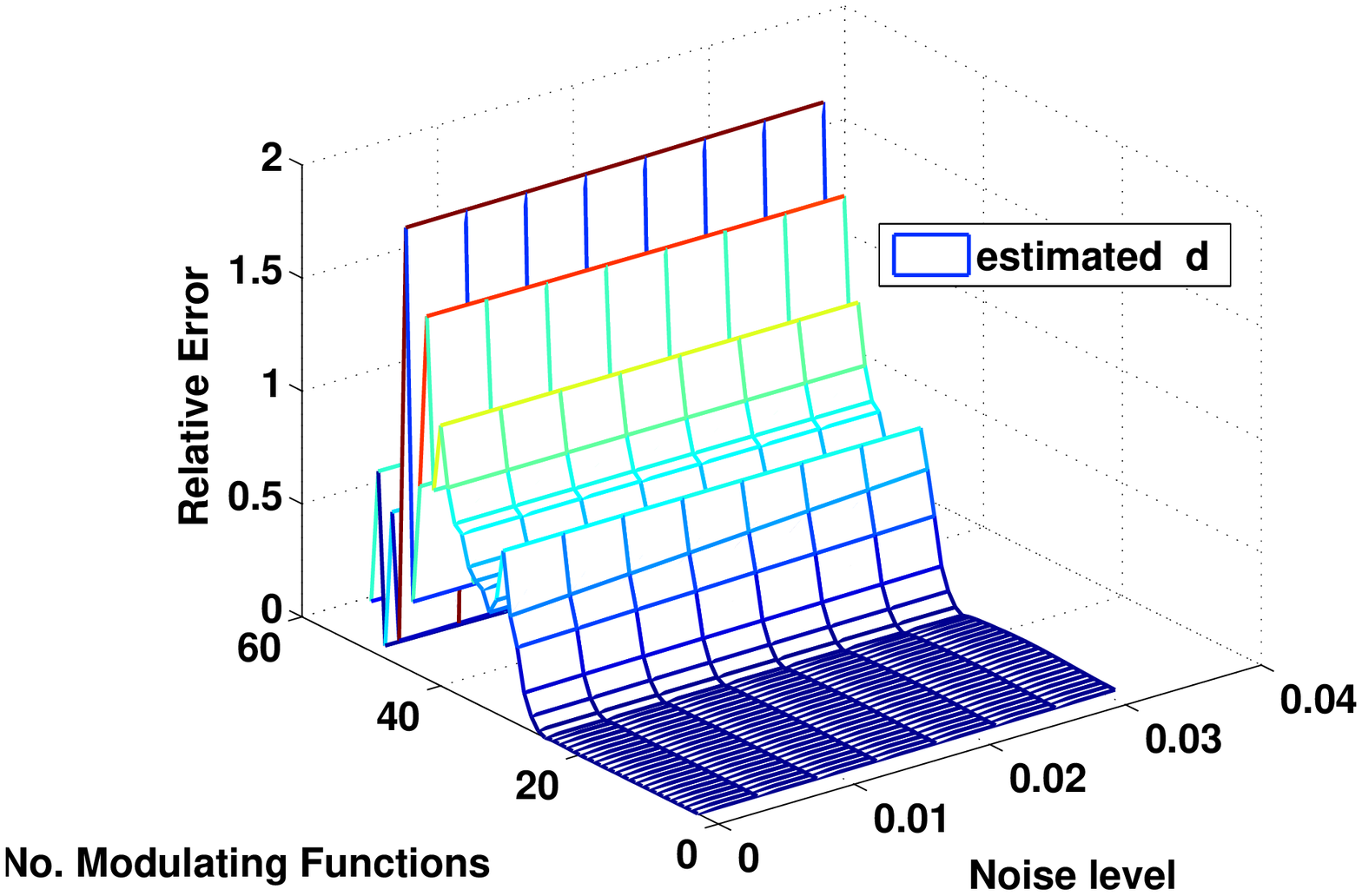}
    \caption{The relative errors of $d$ with different number of modulating functions and different noise levels when $d=1$, $\nu=0.2$.}\label{fig7}
  \end{minipage}
  \hspace{0.5cm}
  \begin{minipage}[b]{0.5\linewidth}
    \centering
    \includegraphics[width=\linewidth]{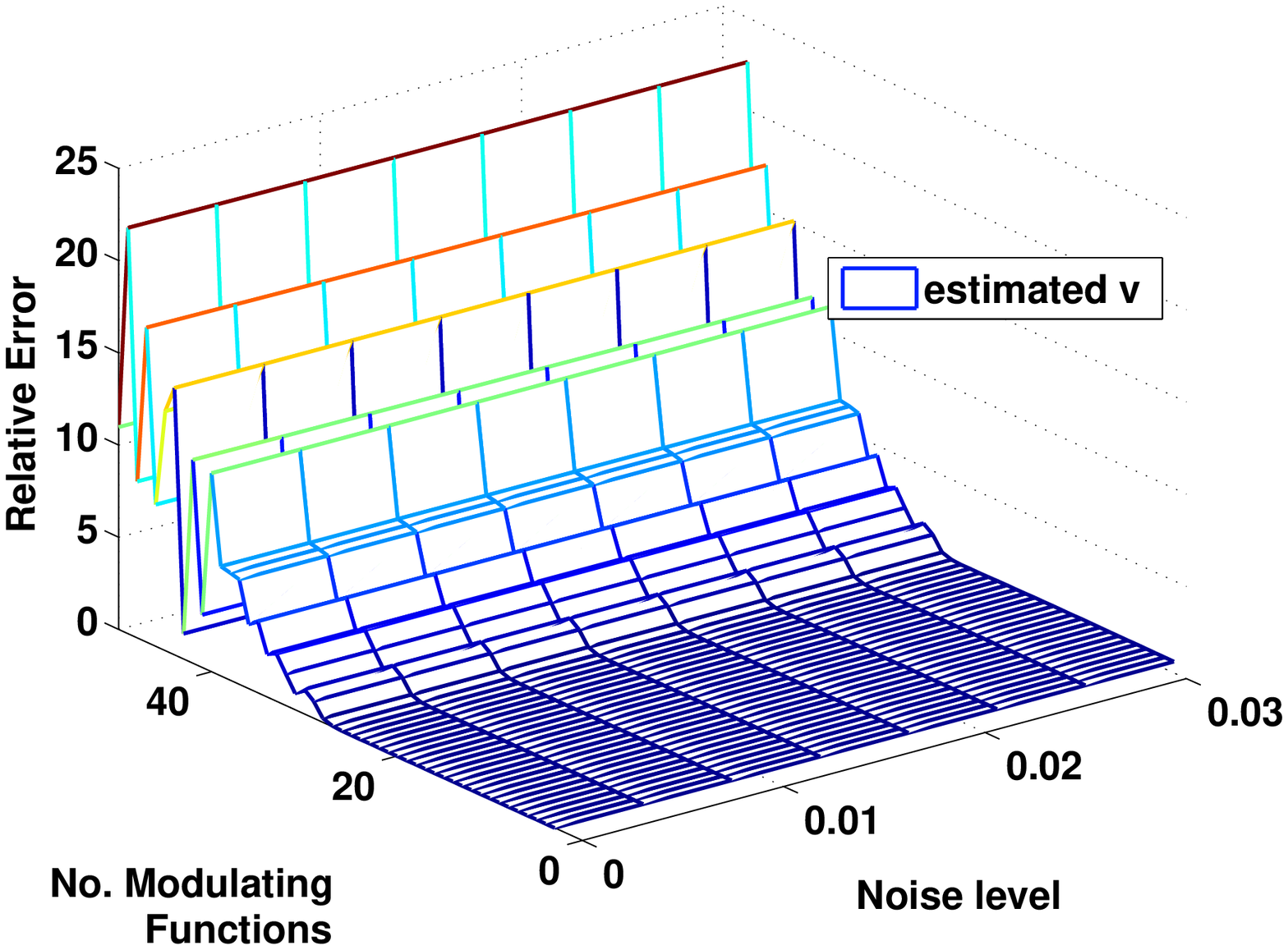}
    \caption{The relative errors of $\nu$ with different number of modulating functions and different noise levels when $d=1$, $\nu=0.2$.}\label{fig8}
  \end{minipage}
\end{figure}

Finally, we briefly comment on the length of the integration interval. In Table \ref{Tab2}, we present the relative errors when estimating  the average velocity and the dispersion coefficient with different integration intervals. We observe that  length of the integration interval has an effect on the accuracy of the performed algorithm. However, the relative errors are still reasonable and less than $1\%$ for $L_1 \ge \frac{2}{3}L$. Furtherer, in Table \ref{Tab3}, we present the relative errors when estimating  all three parameters. It is noted that a larger integration interval is needed. In fact, there is an optimal values for the length of the integration interval and further investigation is needed.

\begin{table}[http]
 \caption{Estimating d and v when the exact values are  $d=1$, $\alpha=1.8$, $\nu=0.2$, and $\Delta x=\frac{1}{1500}$, 3\% noise on both measurements.}\label{Tab2}
\begin{center} \begin{tabular}{c c c ccccc c c} \hline\hline
Number of&  &Relative Errors  & & \\ 
modulating functions&$L_1=5$  &$L_1=6$  &$L_1=7$  &$L_1=8$ &$L_1=8.5$ &$L_1=9$  \\
\hline
3& 8.55E-3  & 2.61E-3 & 1.63E-3  & 1.29E-3 & 1.60E-3 & 2.19E-3\\
4&  7.84E-3  & 2.46E-3  &1.23E-3 & 8.28E-4& 1.08E-3 &1.63E-3  \\
5& 6.77 E-3  & 2.28E-3  &5.86E-4  & 1.17E-4 & 2.79E-4 & 7.68E-4 \\
6& 5.13E-3 & 2.05E-3  &4.33E-4  & 8.66E-4 & 8.38E-4 & 4.03E-4 \\
7& 2.73E-3 & 1.73E-3  & 1.91E-3 &2.12E-3&2.27E-3&1.87E-3\\
\hline \end{tabular} \end{center} \end{table}

\begin{table}[http]
 \caption{When estimating all parameters $d=1$, $\alpha=1.8$, $\nu=0.2$, and $\Delta x=\frac{1}{3500}$, 2\% noise on both measurements.}\label{Tab3}
 \begin{center}
\begin{tabular}{c ccc c c c c c}

\hline\hline
No.  &  &Relative Error   & \\ 
modulating functions &  $L=7.5$ &$L_1=8$&$L_1=9$ \\
\hline

3   & 7.42E-3    &3.64E-3  & 0.442E-3 \\
4   & 6.88E-3 & 2.91E-3 &0.782E-3    \\
5   & 5.82E-3&	2.06E-3 &0.849E-3    \\
6 & 4.39E-3 &	1.19E-3 & 0.572E-3    \\
7& 2.70E-3  &    4.02E-3 &0.248E-3\\

\hline
\end{tabular}
\end{center}
\end{table}

 \section{Conclusion}
In this paper, we have presented a new approach for estimating the parameters  in a space fractional advection dispersion equation. First, we estimated the dispersion coefficient and the average velocity by applying the modulating functions method which transferred the parameter identification problem into solving a system of algebraic equations. Then, the estimations of the unknown parameters have been given by integral formula which are robust against high frequency noises. Then, the modulating functions method combined  with a Newton's type method were used to estimate  the average velocity, the dispersion coefficient and the differentiation order simultaneously,  where the first order derivatives with respect to $\alpha$ of the dispersion coefficient and the average velocity have also been given.  The three parameters optimization problem was transferred into a one parameter optimization problem. Furthermore, numerical simulations have been performed and the results showed the effectiveness of the proposed algorithm. In these numerical examples, we have chosen polynomial modulating functions, which is easy to compute their fractional derivatives. It is mentioned that the choice of the type of modulating functions influences the stability of the linear system and further investigation is needed. In our future work, we aim to generalize the presented method to estimate varying velocity and dispersion coefficients.

\bibliographystyle{plain}



\end{document}